\documentclass{templt}

\newtheorem{theorem}{Theorem} 
\newtheorem{lemma}[theorem]{Lemma}
\newtheorem{cor}[theorem]{Corollary}

\usepackage{hyperref}
\usepackage{orcidlink}
\numberwithin{equation}{section}

\begin{document}

\title[Spectral representation for compact self-adjoint operators]{Spectral theorem for compact self-adjoint operators on smooth Banach spaces}


\author[Mohammed Shameem]{Mohammed Shameem,\orcidlink{0009-0003-2148-3750}}
\address{Department of Mathematics, National Institute of Technology Calicut, Kozhikode, Kerala, India 673601
ORCID: \href{https://orcid.org/0009-0003-2148-3750}{0009-0003-2148-3750}}
\curraddr{}
\email{shameemmaths@gmail.com}
\thanks{}


\subjclass[2020]{Primary 47B01; Secondary 46B07, 47A65}

\date{}

\dedicatory{}

\commby{}

\begin{abstract}
We establish spectral representation results for compact operators that are self-adjoint with respect to the duality mapping on smooth Banach spaces. We obtain a finite spectral decomposition for finite-rank operators and an infinite-series representation for infinite-rank operators under an additional boundedness condition. The result is a natural analogue of the classical spectral theorem for compact self-adjoint operators on Hilbert spaces.
\end{abstract}

\maketitle
Self-adjoint operators on Hilbert spaces are one of the most important classes of operators used in Mathematics and Physics. The spectral representation theorem for compact self-adjoint operators is a fundamental result in the Hilbert space operator theory. It has numerous important consequences, including the singular value decomposition.

Motivated by the classical Hilbert space theory, several authors have investigated notions of self-adjointness in Banach spaces \cite{shtraus1978,wojcik2013}. In this direction, Stampli (1969) introduced adjoint-abelian operators on Banach spaces, which includes self-adjoint operators on Hilbert spaces \cite{stampfli1969}. Later, García-Pacheco (2020) introduced a notion of self-adjoint operators on smooth Banach spaces using the duality mapping \cite{banachSA}. Several properties of these operators were studied and many results on them were established in the article \cite{banachSA,sd2026,stampfli1969}. Certain representation theorems for compact operators in Banach spaces have also been obtained under additional assumptions \cite{edmundcompact,ramesh2025}. The purpose of this paper is to establish a spectral representation theorem for compact self-adjoint operators on smooth Banach spaces that is analogous to the classical Hilbert space theorem.

Throughout the paper, $X$ denotes a complex Banach space, $X^*$ its dual space, and \(B(X)\) the algebra of all bounded linear operators on \(X\). We write \(S_X\) for the unit sphere of \(X\). By an operator we always mean an element of \(B(X)\). The eigenspectrum of an operator \(T\) is denoted by \(\sigma_{\mathrm{eig}}(T)\). The transpose of an operator \(T\in B(X)\) is the operator \(T^{\prime}\) on \(X^*\) defined by
\[
(T'y)(x)=y(Tx), \qquad x\in X,\, y\in X^*.
\]

A Banach space \(X\) is said to be smooth if, for every \(x\in S_X\), there exists a unique \(x^* \in S_{X^*}\) satisfying \(x^*(x)=1\) \cite{dragomir2003}. The associated duality mapping \(J:X \to X^*\) is defined by \(J(x)=x^*\). It is well known that \(J\) is surjective if and only if \(X\) is reflexive, and it is injective if and only if \(X\) is strictly convex \cite{dragomir2003}.

García-Pacheco \cite{banachSA} introduced the following classes of operators on smooth Banach spaces, which are natural analogues of the corresponding notions on Hilbert spaces. An operator \(T\in B(X)\) is called \emph{self-adjoint} if
\[
T' \circ J = J \circ T,
\]
\emph{Hermitian} if
\[
J(x)(Tx)\in\mathbb{R}\quad\text{for all }x\in X,
\]
and \emph{positive} if
\[
J(x)(Tx)\ge0\quad\text{for all }x\in X.
\]
Several fundamental properties of these classes of operators were established in \cite{banachSA}.

\section*{Spectral representation for compact  self-adjoint operators}
In Hilbert space spectral theory, the spectral theorem for compact self-adjoint operators is one of the classical results. It asserts that if \(T\) is a compact self-adjoint operator on a Hilbert space \(H\), then there exist distinct nonzero real eigenvalues \((\lambda_n)_{n=1}^N\), where \(N\in\mathbb{N}\cup\{\infty\}\), such that
\[
T=\displaystyle\sum_{n=1}^{N}\lambda_nP_n,
\]
where \(P_n\) denotes the orthogonal projection of \(H\) onto \(\ker(T-\lambda_nI)\), the kernel of \( T-\lambda_nI\). Moreover, if \(N=\infty\), then \(\lambda_n\to0\) as \(n\to\infty\).

In the Banach space setting, Edmunds \emph{et al.} \cite{edmundcompact,edmunds2010,edmunds2012} established representation theorems for compact operators between reflexive Banach spaces with strictly convex dual spaces. More recently, Ramesh \emph{et al.} \cite{ramesh2025} extended these results to compact operators between reflexive Banach spaces without assuming strict convexity of the dual space. Although these results provide representation formulae for compact operators, they do not yield a spectral representation analogous to the classical Hilbert space spectral theorem. In the present paper, we show that if the underlying Banach space is smooth and the operator is self-adjoint with respect to the duality mapping in the sense of \cite{banachSA}, then such a spectral representation can indeed be obtained. Thus, we establish a natural Banach space analogue of the classical spectral theorem for compact self-adjoint operators on Hilbert spaces.

 In the Hilbert space setting, every self-adjoint operator has real eigenvalues, and eigenvectors corresponding to distinct eigenvalues are orthogonal. García-Pacheco \cite[Theorem~6.3(i)]{banachSA} proved that the eigenvalues of operators that are self-adjoint with respect to the duality mapping on smooth Banach spaces are also real. Here we establish a result similar to the orthogonality of eigenvectors corresponding to distinct eigenvalues of a self-adjoint operator. 
 \begin{lemma}(See Lemma 4 \cite{stampfli1969}, Proposition 6.5 \cite{banachSA})\label{orthogonality-lemma}
     Let $u,v$ be eigenvectors corresponding to distinct eigenvalues of a self-adjoint operator $T$ on a smooth Banach space $X$. Then,  \(J(u)(v)=J(v)(u)=0\).
\begin{proof}
 Let $Tu=\lambda u$ and $Tv=\mu v$, where $\lambda\neq \mu$. Since $\lambda, \mu$ are real,
\[
\lambda J(u)(v)=J(Tu)(v)=J(u)(Tv)=\mu J(u)(v).
\]
As $\lambda\neq \mu$, it follows that $J(u)(v)=0$. Similarly, $J(v)(u)=0$. \end{proof}
 \end{lemma}

Let $V_1$ be a finite-dimensional subspace of a Banach space $X$. Then there exists a closed subspace $V_2$ of $X$ such that $X=V_1\oplus V_2$, where $\oplus$ represents the direct sum of the spaces \cite[Corollary 5.8]{MTNfunctional}. Hence, every $x\in X$ can be uniquely written as $x=v_1+v_2$, where $v_1\in V_1,\,v_2\in V_2$. Then the operator $P\in B(X)$ defined by $P(x)=v_1$ is called the  \textit{natural projection onto} $V_1$ in this article.

\begin{theorem}\label{spectral}
Let $X$ be a smooth Banach space and let $T\in B(X)$ be a compact self-adjoint operator. Let $\lambda_n$ denote eigenvalues of $T$ and $P_n$ is the natural projection onto $\ker(T-\lambda_n I)$. Then:

\begin{enumerate}
    \item If $T$ is of finite rank, then there exists a finite sequence of distinct nonzero real eigenvalues $(\lambda_n)_{n=1}^{N}$, such that
    \[
    T=\sum_{n=1}^{N}\lambda_n P_n,
    \]

    \item If $T$ is of infinite rank,  then there exists an infinite sequence of distinct nonzero real eigenvalues $(\lambda_n)_{n=1}^{\infty}$ of $T$, where $\lambda_n\to0$ as $n\to\infty$. Let  $Q_n=I-\sum_{i=1}^{n}P_i$, then, 
    \[
    T=\big{(}\sum_{i=1}^{n}\lambda_i P_i \big{)}+ TQ_n.
    \]

    \item Moreover, we observe that
    \begin{itemize} 
        \item For each \(n\in\mathbb{N}\), the subspaces \(M_1,\ldots,M_n,N_n\subseteq X\) are \(T\)-invariant, where \(M_j=\operatorname{Range}(P_j)\) and \(N_n=\operatorname{Range}(Q_n)\).
        \item  $X=M_1\oplus M_2\oplus\cdots\oplus M_n\oplus N_n.$
        \item In particular, if \(T\) has rank \(N\), then \(N_m=\ker T\), where $m$ is the number of distinct nonzero eigenvalues of $T$.
    \end{itemize}

    \end{enumerate}
\end{theorem}

    \begin{proof}
    If $T=0$, then the theorem trivially true. So we assume $T\neq 0$. For a compact self-adjoint operator $T \in B(X)$, by Theorem 3.1 of \cite{banachSA}, there exists real eigenvalue $\lambda_1$ for $T$ such that $|\lambda_1|=\|T\|$. Define
    \[
    M_1:=\ker(T-\lambda_1I).
    \]
    Since the eigenspace $M_1$ of $T$ is finite-dimensional, say, of dimension $n_1$, by Auerbach's Lemma (\cite[B.4.8]{opideal}) there exists a basis $\{x^1_1,x^1_2,\ldots,x^1_{n_1}\}$ for $M_1$ and a dual basis   $\{f_1^1,f_2^1,\ldots,f_{n_1}^1\}$ for $M_1^*$ such that $\|x^1_j\|=1=\|f_j^1\|$ and $f^1_j(x^1_i)=\delta^i_j$. Without loss of generality, we can take all $f_j^1$ to be the Hahn-Banach extensions (unique if $X$ is reflexive) of the dual basis elements to the whole space, which retain all these properties. Since $X$ is assumed to be smooth, we can also see that $J(x^1_j)=f^1_j$ for each $j\in \{1,2,\ldots,n_1\}$. Now, define
$$\displaystyle N_1 = \displaystyle\bigcap_{j=1}^{n_1} \ker({f_j^1}).$$
Then $X=M_1\oplus N_1$ (\cite[Corollary 5.8]{MTNfunctional}) and $T(M_1)\subseteq M_1$. Now, for $v \in N_1$, 
    $$
    f^1_j(Tv)=J(x_j^1)(Tv) = J(Tx^1_j)(v) = \lambda J(x^1_j)(v) =0
    $$
    for all $j=1,2,..,n_1$, and hence, $T(N_1)\subseteq N_1$.\vskip0.2cm

    Now consider the truncations $T_2:=T_{N_1}: N_1\to N_1$. If $T_2=0$, then for each $x\in X$, considering $x=y+z\in M_1\oplus N_1$, we get
    \[Tx=T(y)+T_2(z)=T(y)=\lambda_1 y=\lambda_1P_1 x,\]
where $P_1:X\to X$ is the projection onto $M_1$. Thus $T=\lambda_1P_1$ in this case.\vskip0.2cm
Suppose $T_2\neq 0$. Then $T_2$ is clearly a compact operator. To see that it is self-adjoint, consider the duality mapping $J_{N_1}:N_1\to {N_1}^*$. Note that for $x\in N_1$, $J_{N_1}(x)=J(x)_{|_{N_1}}$.
%
%
%
%
Since $T$ is self-adjoint, for $x\in N_1$,
\[  J_{N_1}(T_2x) = J(Tx)|_{N_1} = J(x)T|_{N_1}= J_{N_1}(x)T_2= T_2'J_{N_1}(x).\]

%
%
%
%
%
Considering $N_1$ in place of $X$, the self-adjoint compact operator $T_2:N_1\to N_1$ has a real eigenvalue $\lambda_2$ such that $|\lambda_2|=\|T_2\|\leq |\lambda_1|$, as before. Clearly $\lambda_2\neq \lambda_1$ and using Lemma \ref{orthogonality-lemma}, we get $\ker(T_2-\lambda_2I)=\ker(T-\lambda_2I).$
%
%
%
Call
\[M_2:=\ker(T-\lambda_2I)\]
Let $\{x^2_1,x^2_2,\ldots, x^2_{n_2}\}$ be a basis for $M_2$ and $\{f^2_1,f^2_2,\ldots, f^2_{n_2}\}$ be its dual basis (extended to $X^*$) obtained by using the Auerbach's Lemma. Then, as in the previous case, $J(x^2_i)=f^2_i$ for each $i\in \{1,2,\ldots,n_2\}$. Now define
$$\displaystyle N_2 = \displaystyle\bigcap_{i=1,\, j=1}^{i=2,\, j=n_i} \ker({f_j^i}).$$
Then $X=M_1\oplus M_2\oplus N_2$ and $T(M_i)\subset M_i$ for $i=1,2$.
%
%
%
Since for $i\in \{1,2\}$, $j\in \{1,2,\ldots, n_i\}$, and for any $x\in N_2$, $f_j^i(x)=0$, we have
\[J(x^i_j)(Tx)=JT(x^i_j)(x)=\lambda_iJ(x^i_j)(x)=0,\]
proving $T(N_2)\subset N_2$.
%
%
Hence we can construct $T_3:=T_{|_{N_2}}:N_2\to N_2$ and so on.\vskip0.2cm

Proceeding like this, we get compact self-adjoint operators $T_1,T_2,\ldots$, eigenvalues $\lambda_1, \lambda_2,\ldots$ and closed subspaces $M_1,M_2,\ldots$, $N_1,N_2,\ldots$ such that for each $n\in \mathbb{N}$ and $i\in \{1,2,\ldots, n\}$,
\[X=M_1\oplus M_2 \oplus \ldots \oplus M_n \oplus N_n, \quad T(M_i)\subset M_i, \quad T(N_n) \subset N_n.\]
If $T_{N+1}=0$, for some $N \in \mathbb{N}$, then $N_{N+1}=\ker(T)$, $X=M_1\oplus M_2 \oplus \ldots \oplus M_N \oplus \ker(T)$, and for each $ x= x_1+ x_2+ \ldots + x_N+h\in M_1\oplus M_2 \oplus \ldots \oplus M_N+\ker(T),$
\begin{align*}
    Tx&=T(x_1+x_2+\ldots+x_N+h), \quad \\
    &= \lambda_1 x_1+\lambda_2 x_2+\ldots+\lambda_N x_N\\
    &=\lambda_1 P_1 x+\lambda_2 P_2 x+\ldots+\lambda_N P_N x\\
    &=\displaystyle\sum_{j=1}^N \lambda_j P_j x.
\end{align*}
 In particular, when $T$ is of finite rank, say $N$, then $X_{N+1}=\ker T$, implies $T_{N+1}=0$. Now, suppose $T$ is of infinite rank. If $T_{n+1}\neq 0$ for each $n\in \mathbb{N}$, then we get a sequence of eigenvalues $(\lambda_n)$ of $T$ such that $|\lambda_1|\geq |\lambda_2|\geq \ldots$, bounded below by $0$. 
Suppose  $\{|\lambda_n|\}_{n=1}^\infty$ converges to an $\alpha \neq 0$. Then $B(0,\alpha)^c$ contains infinitely many distinct eigenvalues of $T$, contradicting the fact that for eigenvalues of a compact operator, the only possible accumulation point of nonzero spectrum is $0$ (\cite[Theorem 10.1.2]{MTNfunctional}). Hence in this case, $\lambda_n$ converges to zero. 
Now, take $x\in X$. Writing $X=M_1\oplus M_2 \oplus \ldots \oplus M_n \oplus N_n$, for each $x=y+z$, where $y\in M_1\oplus M_2 \oplus \ldots \oplus M_n$ and $z\in N_n$,
\[\left(T-\displaystyle\sum_{j=1}^n \lambda_j P_j\right) x=\left(T-\displaystyle\sum_{j=1}^n \lambda_j P_j\right) y+\left(T-\displaystyle\sum_{j=1}^n \lambda_j P_j\right) z=Tz = TQ_nx.\]

    \end{proof}

It is not known, in general, whether the sequence \(\{Q_n\}_{n=1}^\infty\) is uniformly bounded. Consequently, obtaining a spectral representation for a compact self-adjoint operator of infinite rank, analogous to the representation established in the finite-rank case, remains an open issue. However, under the additional assumption stated in the following theorem, we obtain a similar spectral representation.

\begin{theorem}\label{Q_nbd}
    Let $T\in B(X)$ be an infinite rank compact self-adjoint operator on smooth Banach space $X$. Let $\lambda_n,P_n,Q_n$ be as in Theorem \ref{spectral}. If $\{Q_nx\}_{n=1}^\infty$ is bounded for each $x \in X$ (or $\{\sum_{j=1}^nP_jx\}_{n=1}^\infty$ is bounded for each $x \in X$), then
    \[
    T=\sum_{n=1}^{\infty}\lambda_n P_n.
    \]
\end{theorem}
\begin{proof}
    Let $\lambda_n,P_n,Q_n$ be as in Theorem \ref{spectral}.
    Thus, for each $x\in X$ with $\|x\|\leq 1$, we have
\[\left\|\left(T-\displaystyle\sum_{j=1}^n \lambda_j P_j\right) x\right\|=\|Tz\|=\|T_{n+1} z\|\leq \|T_{n+1}\| \|z\|\leq |\lambda_{n+1}|\|Q_nx\|,\]
where $Q_n$ is the natural projection onto $N_n$. Since $\{Q_n\}_{n=1}^\infty$ pointwise bounded, it is bounded uniformly and using  $|\lambda_n|\to 0$ as $n\to \infty$, we obtain $T=\displaystyle\sum_{n=1}^{\infty} \lambda_n P_n$, i.e., $T$ converges to $\displaystyle\sum_{n=1}^{\infty} \lambda_n P_n$ in operator norm. If $\{\sum_{j=1}^nP_jx\}_{n=1}^\infty$ is bounded, then using $Q_nx=x-\{\sum_{j=1}^nP_jx\}_{n=1}^\infty$, we get $\{Q_nx\}_{n=1}^\infty$ is bounded for each $x \in X$ and rest of proof follows same.
\end{proof}

Instead of assuming the pointwise boundedness of \(\{Q_n\}_{n=1}^\infty\), a similar spectral representation can be obtained by imposing a condition on the absolute index of \(T\) relative to \(\{\lambda_n\}_{n=1}^\infty\). For the definition of the absolute index, we refer to \cite{Derr1962}. The following corollary illustrates this result.
\begin{cor}
    Let $T\in B(X)$ be an infinite rank compact self-adjoint operator on smooth Banach space $X$ and $\lambda_n,P_n$ be as in Theorem \ref{spectral}. If $T$ has a minimal absolute index $1$, then 
    \[T=\displaystyle\sum_{n=1}^{\infty} \lambda_n P_n.\]
\end{cor}
\begin{proof}
    Using Theorem \ref{spectral}, we have Let $T=\big(\sum_{i=1}^{n}\lambda_iP_i\big)+ TQ_n$, where $P_i$ is the natural projection onto $\ker(T-\lambda_iI)$ and $Q_n=I-\sum_{i=1}^{n}P_i$. Clearly, $T$ is an operator of meromorphic type \cite{Derr1962}. Since $\sigma(T)$ is real by Corollary 1 of \cite{stampfli1969}, it follows from the proof of Theorem 3.1 and Remark 3.2 in \cite{hermspec} that every pole of the resolvent $R(\lambda,T)$ is simple. Hence, by Theorem 7 of \cite{Derr1962}, with $E_n=P_n$ and $F_n=0$, we obtain that $\sum_{n=1}^{\infty}\lambda_nP_n$ converges. Therefore, $Q_n=T-\sum_{i=1}^{n}\lambda_iP_i$ is uniformly bounded, and the result follows from Theorem \ref{Q_nbd}.
\end{proof}
 In the finite-rank case, the construction simplifies considerably and becomes more closely analogous to the Hilbert space spectral theorem, since the natural projections admit a simpler representation. We retain the notations used in the preceding theorem in the following result.



\begin{theorem}
    Let $X$ be a  smooth Banach space and $T \in B(X)$ be a finite rank self-adjoint operator. Then there exist a collection of distinct nonzero eigenvalues $\lambda_1,\lambda_2,\ldots,\lambda_m$, $m\leq N$, of $T$ and a linearly independent collection of corresponding unit eigenvectors $\{x^1_1,..,x^1_{n_1},x^2_1,\ldots, x^{m}_{n
    _{m}}\}$ in $X$ such that
    $$ Tx = \displaystyle\sum \lambda_j J(x_i^j)(x)x_i^j,$$
    for each $x\in X$, with $1\leq j \leq m$, $1\leq i \leq n_{m}$. Moreover, if $M_k=\ker(T-\lambda_k I)$ for $k\in \{1,2,\ldots,m\}$, we have $X=M_1\oplus M_2\oplus \ldots \oplus M_m\oplus N(T)$. In particular,if \(X\) is finite-dimensional, the eigenvectors together with a basis of \(N(T)\) form a basis of \(X\), and hence
    $$x=\displaystyle\sum J(x_i^j)(x)x_i^j$$
    for each $x\in X$, with $1\leq j \leq m+1$, $1\leq i \leq n_{m+1}$
\end{theorem}
\begin{proof}
    Let $X = M_1 \oplus M_2 \oplus .... \oplus M_m \oplus N_m$ be as in the proof of Theorem \ref{spectral}, where $M_k$ stands for the eigenspaces corresponding to nonzero eigenvalues of $T$, for $1\leq k\leq m$. 
    We first observe that $N_m=\ker(T)$. Since $T_{|_{N_m}}=0$ in this case, we get $N_m\subset \ker(T)$. Now, suppose $x\in \ker(T)$. Writing $x=x_1+x_2+\ldots+x_m+h$, where $h\in N_m$, $Tx=0$ implies $\lambda_1 x_1+\lambda_2 x_2+\ldots+\lambda_m x_m=0$ as $h\in N_m$. This forces $x_j=0$ for all $1\leq j \leq m$ as they are independent vectors, giving $x=h\in N_m$.\vskip0.2cm
    
    For $x \in X$, let $P_jx = x_j$, where $x=x_1+...+x_m+h$. For each $j\in \{1,2,\ldots,m\}$, from the proof of Theorem \ref{spectral}, we have 
    $$x_j = \displaystyle\sum_{i=1}^{n_j} J(x_i^j)(x_j)x_i^j.$$
    Since $Th=0$ and using $J(y)(x_j)=0$ for any $y \in M_k$, $k\neq j$, we get

    $$Tx = \sum_{j=1}^k \lambda_jP_jx =  \sum_{{\substack{j\in \{1,2,\ldots,m\}\\ i\in\{1,2,\ldots,n_j\}}}} \lambda_j J(x_i^j)(x)(x_i^j).$$

     Now, let $X$ be finite dimensional. If $N(T)\neq \{0\}$, then we can also choose an Auerbach basis for $N_m$, since it is the eigenspace of the eigenvalue $0$ of $T$, say $\{x^{m+1}_1,x^{m+1}_2,\ldots,x^{m+1}_{n_{m+1}} \}$, and using $J(y)(x_j)=0$ for any $y \in M_k$, $k\neq j$, we get
$$x =\displaystyle\sum_{\substack{j\in \{1,2,\ldots,m+1\}\\ i\in\{1,2,\ldots,n_j\}}} J(x_i^j)(x_j)x_i^j,$$
which proves the claim.
%

\end{proof}
Theorem 10.30 in \cite{rudinfa} states that if the spectrum of $T$ lies on the positive real axis, then $T$ admits a square root whose spectrum also lies on the positive real axis. The following corollary provides a similar result in our setting.

\begin{cor}
Let \(X\) be a Banach space and \(T\in B(X)\) be a finite-rank operator having the representation \(T=\displaystyle\sum_{n=1}^{N}\lambda_nP_n\), where \(\lambda_n\)'s are eigenvalues of \(T\) and \(P_n\) is the natural projection onto \(\ker(T-\lambda_nI)\). Then there exists a compact operator \(S\in B(X)\) such that \(T=S^2\). In fact, if \(T=\displaystyle\sum_{n=1}^{N}\lambda_nP_n\), then \(S=\displaystyle\sum_{n=1}^{N}\sqrt{\lambda_n}\,P_n\). In particular, every finite-rank self-adjoint compact operator on a smooth Banach space admits such a square-root representation.
\end{cor}
 \section*{Acknowledgments}
The author would like to express sincere gratitude to Dr. Deepesh K. P., Assistant Professor at the National Institute of Technology Calicut, for his valuable guidance and support. The author is also grateful to Dr. G. Ramesh, Professor at the Indian Institute of Technology Hyderabad; Dr. M. Veena Sangeetha, Assistant Professor at GITAM School of Science, Hyderabad; and Dr. C. R. Jayanarayanan, Associate Professor at the Indian Institute of Technology Palakkad, for their valuable suggestions and insightful comments. Author gratefully acknowledges the University Grants Commission, India, for financial support towards his research through Fellowship No. 231610086159.

\section*{Statements and Declarations}
The author declares that they have no competing interests.


\bibliographystyle{amsplain}


\end{document}